\documentclass[12pt]{amsart}
\usepackage{amsmath}
\usepackage{amsthm}
\usepackage{amsfonts}
\usepackage{amssymb}
\newtheorem{Theorem}{Theorem}[section]
\newtheorem{Proposition}[Theorem]{Proposition}
\newtheorem{Lemma}[Theorem]{Lemma}
\newtheorem{Corollary}[Theorem]{Corollary}
\theoremstyle{definition}

\theoremstyle{remark}

\numberwithin{equation}{section}

\newcommand{\Z}{{\mathbb Z}}
\newcommand{\R}{{\mathbb R}}
\newcommand{\C}{{\mathbb C}}

\newcommand{\s}[2]{{\langle #1 , #2 \rangle}}
\newcommand{\tr}{{\textrm{\rm tr}\:}}
\newcommand{\ov}[1]{{\overline{ #1}}}
\renewcommand{\Im}{{\textrm{\rm Im}\:}}

\begin{document}

\title[Oscillation theory]{Oscillation theory and semibounded canonical systems}

\author{Christian Remling}

\address{Department of Mathematics\\
University of Oklahoma\\
Norman, OK 73019}
\email{christian.remling@ou.edu}
\urladdr{www.math.ou.edu/$\sim$cremling}

\author{Kyle Scarbrough}

\address{Department of Mathematics\\
University of Oklahoma\\
Norman, OK 73019}

\email{kyle.d.scarbrough-1@ou.edu}
\date{November 15, 2018}



%\thanks{CR's work has been supported by NSF grant DMS 1200553}
\begin{abstract}
Oscillation theory locates the spectrum of a differential equation by counting the zeros of its solutions.
We present a version of this theory for canonical systems $Ju'=-zHu$ and then use it to discuss
semibounded operators from this point of view. Our main new result is a characterization
of systems with purely discrete spectrum in terms of the asymptotics of their coefficient functions; we also
discuss the exponential types of the transfer matrices.
\end{abstract}
\maketitle
\section{Introduction}
A \textit{canonical system }is a differential equation of the form
\begin{equation}
\label{can}
Ju'(x) = -zH(x)u(x) , \quad J=\begin{pmatrix} 0 & -1 \\ 1 & 0 \end{pmatrix} ,
\end{equation}
with a locally integrable coefficient function $H(x)\in\R^{2\times 2}$, $H(x)\ge 0$, $\tr H(x)=1$.
Canonical systems are of fundamental importance in spectral theory because they may be used
to realize arbitrary spectral data; more precisely, they are in one-to-one correspondence
to generalized Herglotz functions, as we will discuss in more detail below.

We usually consider half line problems $x\in [0,\infty)$, and we always impose the boundary condition
\begin{equation}
\label{bc}
u_2(0) = 0
\end{equation}
at the (regular) left endpoint $x=0$. The canonical system together with this boundary condition generates
a self-adjoint relation $\mathcal S$ on the Hilbert space $L^2_H(0,\infty)$ and then also a self-adjoint
operator $S$ on the possibly smaller space $\ov{D(\mathcal S)}$, after dividing out the multi-valued part
$\mathcal S(0)$ of $\mathcal S$. We refer the reader to \cite{Rembook} for more on the basic theory.
We are interested in the spectral theory of $S$.

The \textit{$m$ function }is defined as $m(z)=f(0,z)$ on $z\in\C^+=\{ z\in\C: \Im z>0\}$, and here $f(x,z)$ denotes the
(unique, up to a constant factor) $L^2_H$ solution of \eqref{can}. We also identify the vector
$f(0,z)\in\C^2\setminus \{ 0\}$
with the point $f_1(0,z)/f_2(0,z)\in\C_{\infty}$ on the Riemann sphere, so $m(z)\in\C_{\infty}$.

In fact, the $m$ function is a \textit{generalized Herglotz function: }it is a holomorphic map $m:\C^+\to\C_{\infty}$
that takes values in $\ov{\C^+}$. A (genuine) Herglotz function is defined by the slightly stronger version of
this condition that the values lie in $\C^+$. Such a function satisfies the Herglotz representation formula:
it is of the form
\[
m(z) = a + bz + \int_{-\infty}^{\infty} \left( \frac{1}{t-z} - \frac{t}{t^2+1} \right)\, d\rho(t) ,
\]
with $a\in\R$, $b\ge 0$, and $\rho$ is a positive Borel measure on $\R$ (possibly $\rho=0$) with
$\int\frac{d\rho(t)}{1+t^2}<\infty$. This measure $\rho$
can serve as a spectral measure of $\mathcal S$.

A fundamental result from the inverse spectral theory of canonical systems \cite[Theorem 5.1]{Rembook} says
that every generalized Herglotz function is the $m$ function of a unique canonical system.

A maximal open interval with $H(x)=P_{\alpha}$ there is called a \textit{singular interval }of
\textit{type }$\alpha$, and here
\[
P_{\alpha}= e_{\alpha}e^*_{\alpha} = \begin{pmatrix} \cos^2\alpha & \sin\alpha\cos\alpha \\
\sin\alpha\cos\alpha & \sin^2\alpha \end{pmatrix} , \quad e_{\alpha}=
\begin{pmatrix} \cos\alpha \\ \sin\alpha \end{pmatrix} ,
\]
denotes the projection onto $e_{\alpha}$. Points which are not in the union of the singular intervals
are called \textit{regular. }In the extreme case when $(0,\infty)$ is a single singular interval,
we obtain the $m$ functions $m(z)\equiv a\in\R_{\infty}$;
these are exactly the generalized Herglotz functions that are not Herglotz functions.
These canonical systems $H\equiv P_{\alpha}$ have spectral
measure $\rho=0$, which is consistent with the above remarks
and also with the fact that $D(\mathcal S)=0$ in this case.

\textit{Oscillation theory }is a well known, powerful tool, certainly for the classical equations
such as Schr{\"o}dinger, Sturm-Liouville, Jacobi, Dirac equations. The basic idea is to write solutions
in polar coordinates, and then the angle will satisfy a first order equation, to which comparison principles
can be applied. This will lead to relations between the zeros of solutions and the location of the spectrum.

There is a large literature on oscillation theory in general in a large variety of settings; see, for example,
\cite{GST,GZ,Hart,KT,RBK,Sturm,Swan,WMLN}. However, it appears that
oscillation theory has not yet
been systematically employed in the spectral theory of canonical systems in the way we use it in this paper,
so it will be best for us and the reader to
develop the basic theory from scratch here, relying on these well known ideas and especially
the treatment given in \cite{WMLN}.
The one new aspect that we will have to pay careful attention to will be the presence of relations
(rather than operators) and their multi-valued parts, which correspond to the singular intervals
of our system \cite[Section 2.4]{Rembook}.
When these somewhat tedious technical issues have been addressed,
it will actually turn out that oscillation theory is especially convenient and user-friendly
for canonical systems (compared to, say, Schr{\"o}dinger equations),
thanks to the simple form of the basic equation \eqref{ot}.

We then apply oscillation theory to semibounded canonical systems. In fact, we will almost exclusively restrict
ourselves to systems with specifically $\sigma(H)\subseteq [0,\infty)$, and we denote the collection of these
coefficient functions $H(x)$ by $\mathcal C_+$. Our methods would give more general results, but it seems best
to present them in this setting.

We start out by giving new proofs of the fundamental and beautiful results of Winkler and Woracek \cite{Win,WW}.
We do this for two reasons: first of all, these results certainly deserve some additional exposure;
second, and more importantly, oscillation theory is an ideal tool to analyze these issues, and we believe
that our new proofs are short, direct, and perhaps more transparent than the original proofs, which referred to the
theory of strings as a black box.
Here's what we will actually prove in this part of the paper.
\begin{Theorem}[\cite{WW}]
\label{TWW1}
$H\in\mathcal C_+$ if and only if $H(x)=P_{\varphi(x)}$ for some decreasing function $\varphi(x)$ with
$\pi/2\ge\varphi(0+)\ge\varphi(\infty)\ge -\pi/2$.
\end{Theorem}
As a first minor payoff of our new viewpoint, we effortlessly obtain
a whole line version of Theorem \ref{TWW1}.
\begin{Theorem}
\label{T1.2}
The whole line system with coefficient function $H(x)$, $x\in\R$, has non-negative spectrum
if and only if $H(x)=P_{\varphi(x)}$ for some decreasing function $\varphi(x)$
with $\varphi(-\infty)-\varphi(\infty)\le\pi$.
\end{Theorem}

If $H\in\mathcal C_+$, then the $m$ function
\[
m(z) = a+ bz+ \int_{[0,\infty)} \left( \frac{1}{t-z} - \frac{t}{t^2+1} \right)\, d\rho(t)
\]
can be holomorphically continued to $\C\setminus [0,\infty)$, and $m(t)$ is real valued
and increasing on $(-\infty,0)$. In particular, the limits $m(-\infty), m(0-)\in [-\infty,\infty]$ exist.
\begin{Theorem}[\cite{WW}]
\label{TWW2}
Let $H\in\mathcal C_+$, and write $H(x)=P_{\varphi(x)}$, with $\varphi$ chosen as in Theorem \ref{TWW1}. Then
\[
\tan\varphi(0+)=-m(-\infty), \quad \tan\varphi(\infty)=-m(0-) .
\]
\end{Theorem}

Moving on to the more original parts of the paper, we will then prove the following
characterization of semibounded systems with purely discrete spectrum.
\begin{Theorem}
\label{Tess}
Let $H\in\mathcal C_+$, and write $H(x)=P_{\varphi(x)}$, with $\varphi$ chosen as in Theorem \ref{TWW1}.

Then $\sigma_{ess}(H)=\emptyset$ if and only if
\[
\varphi(x)-\varphi(\infty)= o(1/x) \quad \textrm{as }x\to\infty .
\]
\end{Theorem}
This will actually be a consequence of more general results on the location of the bottom
of the essential spectrum, which we will state and prove in Section 4. These will also imply
part (a) of the following result.
\begin{Theorem}
\label{T1.3}
Let $H\in\mathcal C_+$, and write $H(x)=P_{\varphi(x)}$, with $\varphi$ chosen as in Theorem \ref{TWW1}.

(a) Then $0\in\sigma_{ess}(H)$ if and only if
\[
\limsup_{x\to\infty} x(\varphi(x)-\varphi(\infty)) = \infty .
\]

(b) $0$ is an eigenvalue if and only if $\varphi(x)+\pi/2\in L^2(0,\infty)$.
\end{Theorem}
Part (b) is trivial since the solutions of \eqref{can} at $z=0$ are constant; it is just stated for
completeness here. A combination of both parts of the Theorem gives a description of those
$H\in\mathcal C_+$ whose spectrum starts at zero.

We will also discuss in Section 4 how Theorem \ref{Tess}
contains a new version of Molchanov's \cite{Mol}
well known criterion for the
absence of essential spectrum for a \textit{Schr{\"o}dinger operator }$-d^2/dx^2+V(x)$ as
a special case; see Theorem \ref{T4.3} below for more details.

We then round off our analysis of semibounded canoncial systems by discussing
the exponential orders of the solutions of \eqref{can}, as functions
of $z\in\C$. Here we can be brief since the relevant tools
are all available in the literature \cite{PRW,Rom}, in a slightly different context.

Basically, we will exploit the fact that \eqref{can} for $H\in\mathcal C_+$ can be related to a diagonal
canonical system; this connection is very well known for the smaller class
of Krein strings; see, for example, \cite{KalWW}. We give a direct treatment of this transformation that
never mentions strings explicitly (though of course it is informed by this connection),
and this aspect of our analysis might be of some independent interest also.
The problem of determining the order of a diagonal canonical system has been studied in depth
in \cite{Rom}.

Let's now formulate a result that summarizes the main points.
We define the \textit{transfer matrix }$T(x;z)$ as usual as the $2\times 2$ matrix solution
of \eqref{can} with the initial value $T(0;z)=1$. Its entries are entire functions of $z\in\C$ for each fixed $x\ge 0$,
and one can show that all four entries of $T$ have the same order. Essentially, this will follow from the quotients
being Herglotz functions; see the corresponding part of the
proof of \cite[Theorem 4.19]{Rembook} for a discussion of a very similar statement.

Recall also that the \textit{order }of an entire function $F(z)$ is defined as the infimum of the
$\alpha>0$ for which the estimate $|F(z)|\lesssim \exp(|z|^{\alpha})$ holds.
Clearly, for an arbitrary canonical system, we always have $\textrm{ord}\:T(x;z)\le 1$,
by a simple Gronwall estimate applied to \eqref{can}. Exactly the orders
between $0$ and $1/2$ occur for \textit{semibounded }canonical systems.
\begin{Theorem}
\label{Torder}
Let $H\in\mathcal C_+$, and write $H(x)=P_{\varphi(x)}$ with $\varphi(x)$ chosen as in Theorem \ref{TWW1}.

(a) $\textrm{\rm ord}\: T(x;z)\le 1/2$ for all $x\ge 0$.

(b) Conversely, for any $0\le\nu\le 1/2$, there are semibounded canonical systems $H\in\mathcal C_+$
with $\textrm{\rm ord}\:T(x;z;H)=\nu$ for some $x>0$.

(c) If $\textrm{\rm ord}\: T(L;z)<1/2$, then $\varphi'(x)=0$ for almost every $x\in (0,L)$.
\end{Theorem}
Recall that $\varphi$ is a decreasing function, so will be differentiable at almost every $x$.
Since the pointwise derivative computes the Radon-Nikodym derivative of the absolutely continuous part
of the measure $-d\varphi$,
another way of stating part (c) is to say that this measure must be purely singular on $(0,L)$
if $\textrm{\rm ord}\: T(L;z)<1/2$.

One can in principle go beyond this by referring to \cite[Theorem 2]{Rom}, but this will become
intricate and the resulting criteria will probably not be easy to check for a given $\varphi$.
What we have stated here will be comparatively easy to prove, and we present these arguments
in Section 5. We will also give an easy direct argument for part (b), which will not
depend on \cite[Theorem 2]{Rom}.
\section{Oscillation theory}
Given a non-trivial solution $u$ of \eqref{can} for $z=t\in\R$, introduce $R(x)>0$, $\theta(x)$ by writing
$u=Re_{\theta}$, with $\theta(x)$ continuous and, as above, $e_{\theta}=(\cos\theta,\sin\theta)^t$.
Then the \textit{Pr{\"u}fer angle }$\theta(x)$ is in fact absolutely continuous and solves
\begin{equation}
\label{ot}
\theta'(x) = te^*_{\theta(x)}H(x)e_{\theta(x)} .
\end{equation}

We will also consider the problems on bounded intervals $[0,L]$, and then we impose the boundary condition
\begin{equation}
\label{bcbeta}
e^*_{\beta}Ju(L) = u_1(L) \sin\beta - u_2(L)\cos\beta = 0
\end{equation}
at $x=L$, with $0\le\beta<\pi$. This, together with the boundary condition \eqref{bc} at $x=0$, defines
a self-adjoint relation $\mathcal S_L^{(\beta)}$ on $L^2_H(0,L)$; see again \cite[Chapter 2]{Rembook} for more details.
\begin{Proposition}
\label{P2.1}
Let $\theta(x;t)$ be a solution of \eqref{ot} with $t$ independent initial value
$\theta(0;t)=\alpha$.
Then $\theta(x;t)$ is an increasing function of $t\in\R$, and as a function of $x\ge 0$,
the Pr{\"u}fer angle $\theta(x;t)$ is increasing if $t\ge 0$ and decreasing if $t\le 0$.

In fact, $t\mapsto\theta(x;t)$ is strictly increasing for $x>0$ unless $(0,x)$ is contained in a singular interval
of type $\alpha+\pi/2$.
\end{Proposition}
\begin{proof}
The first few claims are immediate from \eqref{ot}; for the montonicity in $t$, we refer to the comparison principle
\cite[Section III.4]{Hart} for first order ODEs.

If $t\mapsto \theta(L;t)$ were constant on some interval $a\le t\le b$, for some $L>0$, then the
corresponding solutions $u(x;t)$ would be candidate eigenfunctions, with eigenvalue $t$, of the problem
on $(0,L)$ with boundary condition $\beta\equiv\theta(L;a)\bmod \pi$ at $x=L$.
A contradiction can only be avoided if $Hu=0$ on $(0,L)$ for these $u$, and this
makes $H=P_{\alpha+\pi/2}$ there.
\end{proof}
By this monotonicity, the Pr{\"u}fer angle $\theta(L;t)$ can be used to count how many times
the boundary condition \eqref{bcbeta} was satisfied. This in turn lets us locate the spectrum.

We start with the problem on a bounded interval $[0,L]$, with boundary condition \eqref{bcbeta}.
We denote the spectral projections of the associated self-adjoint operator $S_L^{(\beta)}$
(extracted from the relation $\mathcal S_L^{(\beta)}$ by dividing out the multi-valued part) by $E_L^{(\beta)}$,
and we use the short-hand notation
$\dim P$ for what is really the dimension of the \textit{range }of the projection $P$.
We will also write $E(s,t)$ instead of the more precise $E((s,t))$,
and similarly for other types of intervals, to avoid an aesthetically offensive proliferation
of parentheses.
\begin{Lemma}
\label{L2.1}
Let $\theta(x;t)$ be the solution of \eqref{ot} with $\theta(0;t)=0$. Then
\[
\dim E_L^{(\beta)}[s,t) = \left\lceil \frac{1}{\pi} \left( \theta(L;t)-\beta \right) \right\rceil -
\left\lceil \frac{1}{\pi} \left( \theta(L;s)-\beta \right) \right\rceil .
\]
\end{Lemma}
The dimension of the spectral projection of course equals the number of eigenvalues
in $[s,t)$.
\begin{proof}
The eigenvalues $\lambda$ are characterized by the condition $\theta(L;\lambda)\equiv\beta \bmod\pi$.
Now the monotonicity and continuity of $t\mapsto \theta(L;t)$ make it clear that
$\lceil (\theta(L;t) -\beta)/\pi\rceil$ jumps
by $1$ at each eigenvalue and is constant on the intervals between those.

This argument does not literally apply when $(0,L)$ is a singular interval of type $\pi/2$, but this scenario is trivial
and the claim can then be checked directly; all spectral projections are zero in this case.
\end{proof}
\begin{Theorem}
\label{T2.1}
Suppose that $(0,\infty)$ does not end with a singular half line $(L,\infty)$, write $E$
for the spectral projection of the half line operator, and
let $\theta(x;t)$ be the solution of \eqref{ot} with $\theta(0;t)=0$. Then
\begin{equation}
\label{2.3}
\dim E(s,t) = \lim_{L\to\infty} \left\lfloor \frac{1}{\pi} \left( \theta(L;t)-\theta(L;s)\right)
\right\rfloor .
\end{equation}
\end{Theorem}
The existence of the limit, with the understanding that it may equal infinity,
is part of the statement.

If $(0,\infty)$ does end with a singular half line $(L,\infty)$ of type $\gamma$, say,
then we are effectively dealing with the problem on $(0,L)$ with boundary condition $\beta=\gamma+\pi/2$
at $x=L$ \cite[Theorem 3.18]{Rembook}, so we are back in the case already dealt with in Lemma \ref{L2.1}.
\begin{proof}
Let's abbreviate the expression from the statement by
\[
F(L) = \frac{1}{\pi} \left( \theta(L;t)-\theta(L;s)\right) .
\]
We will establish the following two inequalities:
\begin{align}
\label{2.1}
& \lfloor F(L) \rfloor \le \dim E(s,t)\quad
\textrm{for all }L>0 ;\\
\label{2.2}
& \dim E(s,t) \le \liminf_{L\to\infty} \lceil F(L) \rceil - 1 .
\end{align}

Let's first check that these inequalities will imply \eqref{2.3}: clearly,
\begin{align*}
\liminf_{L\to\infty} \lceil F(L) \rceil - 1 & \le \limsup_{L\to\infty} \lceil F(L) \rceil - 1 \\
& \le \sup_{L>0} \lceil F(L) \rceil - 1 \le \sup_{L>0} \lfloor F(L) \rfloor ,
\end{align*}
so we have equality throughout here. In particular, $\lim_{L\to\infty} \lceil F(L) \rceil$
exists, and it then follows that $\lfloor F(L) \rfloor$ converges as well: this is immediately clear
if $F(L)\notin\Z$ for all large $L$, and if $F(L_n)\in\Z$ for some sequence $L_n\to\infty$,
then $F(L_n)\to\infty$, or we would obtain a contradiction to our inequalities (a direct proof of this fact
is also possible).

So it suffices to establish the inequalities, and we start with \eqref{2.1}. Given $L>0$,
define $\beta\in [0,\pi)$ by writing $\theta(L;t)=n\pi+\beta$, $n\in\Z$. Our intention here is to choose the
boundary condition that makes $t$ an eigenvalue of the problem on $[0,L]$, but actually there is an
exceptional case: if $H\equiv P_{e_2}$
on $(0,L)$, then $Hu=0$ there. This scenario, however, is completely
trivial because now $F(L)=0$, and we can ignore it. Lemma \ref{L2.1} then shows that
\[
\dim E_L^{(\beta)}[s,t] = 1 + n - \left\lceil \frac{1}{\pi} \left( \theta(L;s)-\beta\right) \right\rceil =
\lfloor F(L) \rfloor + 1 .
\]
Now we adapt the arguments presented in \cite[Chapter 14]{WMLN}. Suppose that \eqref{2.1} failed. Then
\begin{equation}
\label{2.4}
\dim \mathcal M \ge 2, \quad \mathcal M = R(E_L^{(\beta)}[s,t]) \ominus R(E(s,t)) ;
\end{equation}
of course, this definition of $\mathcal M$ does not make strict formal sense
if taken at face value since the projections
act in different Hilbert spaces. We really identify $R(E_L^{(\beta)})\subseteq L^2_H(0,L)$ with
a subspace of $L^2_H(0,\infty)$ in the obvious way, by extending elements of this space by the
zero function on $(L,\infty)$. In the same way, the self-adjoint relation $\mathcal S_L^{(\beta)}$
can be thought of as a relation on $L^2_H(0,\infty)$.

Since we are projecting onto a bounded interval, the elements of $R(E_L^{(\beta)}[s,t])$
are contained in $D(\mathcal S_L^{(\beta)})$, the domain of the self-adjoint relation. If we
take such elements $(f,g)\in\mathcal S_L^{(\beta)}$, then the standard
representatives $f(x)$ of $f\in L^2_H(0,L)$, determined as in \cite[Lemma 2.1]{Rembook},
will satisfy the boundary condition \eqref{bcbeta} at $x=L$. Now \eqref{2.4} implies
that there is a non-zero element $f\in\mathcal M$ with $f(L)=0$. This element, again extended by zero
beyond $L$ and viewed as an element of $L^2_H(0,\infty)$, will lie in $D(\mathcal S)$, the domain of
the self-adjoint relation on the half line $(0,\infty)$.

We can now evaluate $g-cf$, with $c=(s+t)/2$ and $g=S_L^{(\beta)}f$, the image of $f$ under
the \textit{operator }$S_L^{(\beta)}$, in two ways: if we work on $(0,L)$, then, since
$f=E_L^{(\beta)}[s,t]f$, we obtain $\|g-cf\|\le (t-s)/2 \|f\|$.
On the other hand, we can also view $(f,g)\in\mathcal S$ as an element of the self-adjoint
relation $\mathcal S$ on the half line, after extending both functions by zero for $x>L$, as usual.
Then $g=Sf+h$ with $h\in\mathcal S(0)$, the multi-valued part of $\mathcal S$; we cannot be sure here
if $g$ is still the operator image of $f$ (though this will follow when $L$ is regular). However,
we do know that $f,Sf\in\ov{D(\mathcal S)}=\mathcal S(0)^{\perp}$, so
\[
\|g-cf\|^2 \ge \|(S-c)f\|^2 \ge \left( \frac{t-s}{2} \right)^2 \|f\|^2 ;
\]
to obtain the second estimate, we have used that $E(s,t)f=0$.

So we in fact have equality here, but then it follows, by functional calculus again, that
$f=E(\{s,t\})f$ must be linear combination of the eigenfunctions for the eigenvalues $s,t$,
so let's write $f=u_s+u_t$, and here $u_{\lambda}$ solves $Ju'_{\lambda}=-\lambda Hu_{\lambda}$.
The corresponding representative $f(x)=u_s(x)+u_t(x)$, built from these solutions, is absolutely
continuous, satisfies $Jf'=-Hg$, with $g=su_s+tu_t\in L^2_H$, and represents the zero element
of $L^2_H$ on $(L,\infty)$. Now \cite[Lemma 2.26]{Rembook}, applied to this interval, shows
that $f(c)=0$ at all regular points $c>L$. Since $(L,\infty)$ is not contained in a singular
half line, by our assumption, there are such regular points $c>L$. Fix one, and observe that then
$u_s(c)=-u_t(c)$ satisfy the same boundary condition at $x=c$. So $u_s,u_t$ are orthogonal on
$(c,\infty)$, being eigenfunctions belonging to different eigenvalues. Since $\|f\|_{L^2_H(c,\infty)}=0$,
this implies that $u_s, u_t$ also have zero norm on $(c,\infty)$, but for a non-zero solution this is only possible
if $(c,\infty)$ were contained in a singular half line.
This contradiction establishes \eqref{2.1}.

The proof of \eqref{2.2} is, fortunately, less involved technically. We can assume that
$\liminf \lceil F(L)\rceil <\infty$. Pick a sequence $L_n\to\infty$
with $\lceil F(L_n)\rceil = \liminf \lceil F(L)\rceil$. Define $\beta_n\in [0,\pi)$ by writing
$\theta(L_n;s)=N_n\pi + \beta_n$, that is, we choose the boundary condition that makes $s$ an eigenvalue
of the problem on $[0,L_n]$. The exceptional situation that was already briefly mentioned above will
not occur here for large $n$ because then $H(x)$ will not be identically equal to $P_{e_2}$ on $(0,L_n)$.

The boundary condition $\beta_n$ can be implemented by a singular half line
$(L_n,\infty)$ of type $\beta_n+\pi/2$. These modified canonical systems
\[
H_n(x) = \begin{cases} H(x) & x<L_n \\ P_{\beta_n+\pi/2} & x>L_n \end{cases}
\]
converge to $H$ as $n\to\infty$ with respect to the metric discussed in \cite[Section 5.2]{Rembook}.
Moreover, in general, convergence in this metric is equivalent to the locally uniform (on $\C^+$)
convergence of the associated $m$ functions \cite[Theorem 5.7(b), Corollary 5.8]{Rembook}, and this
in turn implies that the spectral measures $\rho_n$ converge to $\rho$ in weak $*$ sense.
Thus it now suffices to show that
\[
\dim E_n(s,t)\le \lceil F(L_n) \rceil - 1 .
\]
This, with equality, is an immediate consequence of Lemma \ref{L2.1}; recall here that $\dim E_n(\{ s\} )=1$
by the choice of $\beta_n$.
\end{proof}

As usual, these results also tell us where the essential spectrum starts, because this is the point
where spectral projections become infinite dimensional. We don't state general results of this type
here, but we will see these methods in action in Section 4.
\section{Semibounded canonical systems}
In this section, we prove Theorems \ref{TWW1}, \ref{TWW2}, and \ref{T1.2},
in this order.
\begin{proof}[Proof of Theorem \ref{TWW1}]
We want to give an oscillation theoretic treatment, so we start out by observing that the
condition that $H\in\mathcal C_+$ is of course equivalent to
\begin{equation}
\label{3.1}
E(-t,0)=0 \quad \textrm{for all }t>0 .
\end{equation}
Let $\theta(x;t)$ again be the solution of
\eqref{ot} with initial value $\theta(0;t)=0$. Since $\theta(x;0)=0$, Theorem \ref{T2.1} shows that
\eqref{3.1} is equivalent to
\begin{equation}
\label{3.2}
\theta(x;-t)>-\pi \quad \textrm{for all }x,t>0 .
\end{equation}
This also holds when $(0,\infty)$ ends with
a singular half line $(L,\infty)$ of type $\beta+\pi/2$, say, with $0\le\beta<\pi$ (so we effectively
have the problem on $(0,L)$, with boundary condition $\beta$ at $x=L$). In this case,
we refer to Lemma \ref{L2.1} directly. This produces the stronger looking bound $\theta(x;-t)>-\pi+\beta$,
but actually this is implied by \eqref{3.2} in the current situation, for the following reason: if we had
$\theta(a;-t)\in (-\pi,-\pi+\beta]$ for some $a\ge L$,
then also $\theta(a;-t')\in (-\pi,-\pi+\beta)$ for suitable $t'>t$, but then
$\lim_{x\to\infty}\theta(x;-t')=-2\pi+\beta<-\pi$.

Suppose now that $H\in\mathcal C_+$, or, equivalently, that
\eqref{3.2} holds. We first claim that then $\det H(x)=0$ for almost every $x>0$.
This is obvious from \eqref{ot} since for any $\theta$, we have $e^*_{\theta}H(x)e_{\theta}\ge \det H(x)$,
so clearly \eqref{3.2} will fail for large $t$ and $x$ if $\det H(x)>0$ on a set of positive measure.

We can thus write $H(x)=P_{\varphi(x)}$, for some function $\varphi(x)$, and we now claim that we can take
\begin{equation}
\label{3.6}
\varphi(x) =\varphi_0(x), \quad \varphi_0(x) = \lim_{t\to\infty} \theta(x;-t) + \frac{\pi}{2} ,
\end{equation}
here. The limit defining $\varphi_0$ exists since $\theta(x;-t)>-\pi$ is a decreasing function of $t>0$, and the
monotonicity of $\theta$ in $x$ and \eqref{3.2} will then show that $\varphi_0(x)$ has the stated properties,
so it suffices to prove \eqref{3.6}.

For $H(x)=P_{\varphi(x)}$, we can write \eqref{ot} in the form
\begin{equation}
\label{3.5}
\theta' = -t \sin^2(\theta-\psi(x)) , \quad \psi(x)= \varphi(x)-\frac{\pi}{2} .
\end{equation}
Integration of this gives
\[
\int_0^L \sin^2\left( \theta(x;-t)-\varphi(x)+\frac{\pi}{2}\right) \, dx = -\frac{\theta(L;-t)}{t} < \frac{\pi}{t} .
\]
Since this holds for all $L>0$,
Fatou's lemma now shows that
\[
\liminf_{t\to\infty} \sin^2\left( \theta(x;-t)-\varphi(x)+\frac{\pi}{2} \right) = 0
\]
for almost all $x>0$, or, equivalently, $\varphi_0(x)\equiv \varphi(x)\bmod \pi$ almost everywhere.
Since $P_{\alpha+n\pi}=P_{\alpha}$, this establishes \eqref{3.6}.

Conversely, suppose now that $H(x)=P_{\varphi(x)}$, with $\varphi(x)$ as described in the Theorem.
We must show that then \eqref{3.2} holds.

The idea behind our argument is simple: both functions $\theta(x), \psi(x)=\varphi(x)-\pi/2$
are decreasing, and initially $\theta(0)\ge\psi(0+)$. Now the form of \eqref{3.5} will guarantee that
$\theta$ can never overtake $\psi$, and $\psi$ stops at the value $-\pi$ at the latest.

The essence of the method is best seen by first considering the simpler case when $\psi(0+)<0=\theta(0)$. If
\[
y:=\sup \{b>0: \theta(x)>\psi(x-) \:\: \textrm{\rm on }0<x<b \}
\]
were finite, then
$\theta(y)=\psi(y-)$. On a suitable interval $x\in (a,y)$, we have the estimate
\[
\sin^2(\theta-\psi(x))\le (\theta-\psi(y-))^2 ,
\]
as long as $\theta(a)\ge \theta\ge \psi(x)$. However, the solution $\theta_1$ of
\[
\theta'_1=-t(\theta_1-\psi(y-))^2 , \quad
\theta_1(a)=\theta(a)>\psi(y-) ,
\]
will not reach $\psi(y-)$ in finite time, so we obtain a contradiction to
the comparison principle. Thus $y=\infty$, and this says that $\theta(x)>\psi(x-)$ for all $x>0$, and
then \eqref{3.2} is an immediate consequence.

These arguments could also handle the case when $\psi(0+)=0$, but it is technically more convenient
to then approximate $H(x)=P_{\varphi(x)}$ by the coefficient functions
\[
H_n(x) = \begin{cases} H(x) & x> 1/n \\ P_{\varphi(1/n+)} & x<1/n \end{cases} .
\]
These will converge to $H$ with respect to the metric mentioned above and, what is more important right now, this
will give us the weak $*$ convergence of the spectral measures.

So if $\psi(x)<0$ for all $x>0$, then it will follow that $H\in\mathcal C_+$, by the case already covered.
This only leaves the case of an initial singular interval of type $\pi/2$, but this can be removed without
changing the spectral measure, and thus we are done in this case also.
\end{proof}
A more general version of Theorem \ref{TWW1}, also due to Winkler-Woracek \cite{WW}, can be established
by the same arguments, with only very minor adjustments, which we leave to the reader.
\begin{Theorem}
\label{T3.1}
The negative spectrum
$\sigma(H)\cap (-\infty,0)$ consists of at most $N$ points if and only if $H(x)=P_{\varphi(x)}$
for some decreasing function $\varphi(x)$ with
\[
\frac{\pi}{2} \ge\varphi(0+)\ge\varphi(\infty)\ge -N\pi -\frac{\pi}{2} .
\]
\end{Theorem}
This, in turn, gives the following characterization of the larger class of coefficient functions of this type,
but with a possibly unbounded $\varphi(x)$.
\begin{Corollary}
\label{C3.1}
(a) $H(x)=P_{\varphi(x)}$ for some decreasing function $\varphi(x)$ with $\varphi(0+)<\infty$
if and only if the problems
on $[0,L]$ have finite negative spectrum for all $L>0$.

(b) If $\sigma(H)\subseteq [c,\infty)$ for some $c\in\R$, then $H(x)=P_{\varphi(x)}$ for some decreasing
function $\varphi(x)$ with $\varphi(0+)<\infty$.
\end{Corollary}
To prove part (a), just recall that boundary conditions at $x=L$ can be implemented by a singular half
line $(L,\infty)$. This will then imply part (b), after establishing the easy fact that problems on $(0,L)$
will be semibounded if the half line problem has this property.

The converse of part (b) is false, and counterexamples are provided by Schr{\"o}dinger operators that are
unbounded below, when these are written as canonical systems.

\begin{proof}[Proof of Theorem \ref{TWW2}]
It will be convenient to also express the values of $m(-t)$, $t>0$, in terms of an angle, so write
$m(-t)=\cot\alpha(-t)$, with $-\pi<\alpha(-t)<0$. Here, we again leave the trivial case case
$H(x)\equiv P_{e_2}$ to the reader. We also write $\psi(x)=\varphi(x)-\pi/2$, as above.
We then want to show that $\psi(0+)=\alpha(-\infty)$, $\psi(\infty)=\alpha(0-)$.

The key tool will be the following fact.
\begin{Lemma}
\label{L2.2}
Let $H\in\mathcal C_+$, and let $\theta(x;-t)$, $t>0$, be the solution of \eqref{ot} with $\theta(0;-t)=\alpha(-t)$.
Then $\theta(x;-t)\ge -\pi$ for all $x\ge 0$.
\end{Lemma}
\begin{proof}
The initial value of the solution $f=Re_{\theta}$ of \eqref{can} with Pr{\"u}fer angle $\theta$
is a multiple of $(m(-t),1)^t$, so
$f\in L^2_H(0,\infty)$. Suppose now that $\theta(L;-t)=-\pi$ for some $L>0$. This says that
$f(L)=e_1$, after multiplying by a suitable (negative) constant. The modified version of this solution
\[
f_L(x) = \begin{cases} e_1 & x<L \\ f(x) & x>L \end{cases}
\]
lies in $D(\mathcal S)$, the domain of the self-adjoint relation on $(0,\infty)$. More specifically,
$(f_L, g_L)\in\mathcal S$, with
\[
g_L(x) = \begin{cases} 0 & x<L \\ -tf(x) & x>L \end{cases} .
\]
If we denote the self-adjoint \textit{operator }by $S$, then
\[
\s{f_L}{g_L} = \s{f_L}{Sf_L} .
\]
Note that this will hold even though $g_L$ need not equal $Sf_L$ since even in that case
$g_L$ differs from this operator image
by at most an element of the multi-valued part $\mathcal S(0)$,
and $f_L\in D(\mathcal S) \subseteq \mathcal S(0)^{\perp}$.

Now $\s{f_L}{Sf_L}\ge 0$ by functional calculus, but on the other hand,
\[
\s{f_L}{g_L} = -t \int_L^{\infty} f^*(x)H(x)f(x)\, dx \le 0 .
\]
So this last integral equals zero, but this means that $Hf=0$ almost everywhere on $(L,\infty)$, and thus
$f(x)=e_1$ and $\theta(x;-t)=-\pi$
on $x\ge L$.
\end{proof}

Let's now return to the proof of Theorem \ref{TWW2}.
We first show that $\alpha(-\infty)\ge \psi(0+)$.
If this were false, then the Pr{\"u}fer angle $\theta(x;-t)$ with
the initial value $\theta(0;-t)=\alpha(-t)$ from Lemma \ref{L2.2}
would satisfy $\theta(x;-t)\le\psi(x)-\delta$
on some interval $x\in (0,a)$ for all large $t>0$. But now \eqref{3.5} shows that then $\theta'\le -t\sin^2\delta$
there, as long as $\theta-\psi\ge -\pi+\delta$. It follows that $\theta(x;-t)$ will decrease
beyond $-\pi$ for large $t$, contrary to what we established in Lemma \ref{L2.2}. Recall also in this
context that we already dismissed the case $\psi\equiv 0$, so we will have $\psi(x)<0$ for all large $x$.

On the other hand, $\alpha(-\infty)>\psi(0+)$ is also impossible, and the argument is similar. We could then
pick first a sufficiently large $t_1>0$ and then $a>0$ such that $\theta(a;-t_1)> \psi(0+)$ also.
Here, $\theta$ again refers to the Pr{\"u}fer angle from
Lemma \ref{L2.2}, with initial value $\theta(0;-t)=\alpha(-t)$. Again, \eqref{3.5} shows that
$|\theta'(x;-t_2)|$ can be made arbitrarily large on $0\le x\le a$ by sending $t_2\to\infty$,
at least as long as $\theta(x;-t_2)$ stays at some distance from $\psi(0+)$.
This means that $\theta(a;-t_2)$ will have overtaken $\theta(a;-t_1)$ for all large $t_2\gg t_1$,
but this contradicts the monotonicity of $m(-t)$ on $t>0$. More explicitly,
$\cot\theta(a;-t)=m_a(-t)$ is the $m$ function
of the problem on $(a,\infty)$, and $H(x+a)\in\mathcal C_+$ also, by Lemma \ref{L2.2}
and its proof. Thus it is not
possible that $\theta(a;-t_2)<\theta(a;-t_1)$ for $t_2>t_1$.

Next, we show that $\psi(\infty)\le \alpha(0-)$. We again consider the Pr{\"u}fer angles
with the initial values from Lemma \ref{L2.2}. By \eqref{3.5},
$\theta(x;-t)$ can only approach a value that is $\equiv\psi(\infty)\bmod \pi$ when $x\to\infty$.
Now if we had $\psi(\infty)>\alpha(0-)$, then also $\psi(\infty)>\theta(0;-t)$ for sufficiently small $t>0$,
so the first value at which we can stabilize is $\psi(\infty)-\pi$.
However, by Lemma \ref{L2.2}, we also must not cross the value $-\pi$, and since
$\psi(\infty)\in [-\pi,0]$, this forces $\psi(\infty)=0$, but this puts us back in the trivial case
$\psi(x)\equiv 0$ that we already dispensed with.

Finally, we must rule out the situation where $\psi(\infty)<\alpha(0-)$. In this case,
we can rotate all angles by $\gamma=-\pi-\alpha(0-)$; in other words, we move
$\alpha(0-)$ to its new destination $-\pi$.

This can be implemented by letting the rotation matrix
\[
R_{\gamma} = \begin{pmatrix} \cos\gamma & -\sin\gamma \\ \sin\gamma & \cos\gamma
\end{pmatrix}
\]
act on $m$ as a linear fractional transformation $m_{\gamma}=R_{\gamma}m$, and this is the same as conjugating
the coefficient function $H_{\gamma}(x)=R_{\gamma}H(x)R_{-\gamma}$ \cite[Theorem 3.20]{Rembook}.
By inspecting
\[
m_{\gamma}(z) = \frac{m(z) \cos\gamma - \sin\gamma}{m(z)\sin\gamma + \cos\gamma} ,
\]
we see that our choice of $\gamma$ makes sure that $m_{\gamma}$ is still holomorphic
on a neighborhood of $(-\infty,0)$,
so $H_{\gamma}\in \mathcal C_+$ as well. By its construction, the angle functions $\alpha_{\gamma},\psi_{\gamma}$
of the new system are simply the rotated versions $\alpha+\gamma$, $\psi +\gamma$ of the old ones.
However, now we obtain a contradiction to Lemma \ref{L2.2} because $\psi(\infty)+\gamma <-\pi$ has
been moved past $-\pi$, but $\alpha(-t)+\gamma>-\pi$ for $t>0$, so the Pr{\"u}fer angle $\theta_{\gamma}(x;-t)$
would have to cross the forbidden value $-\pi$ before it can stabilize.
\end{proof}

\begin{proof}[Proof of Theorem \ref{T1.2}]
Assume that $\sigma (H)\subseteq [0,\infty)$. In general, the essential spectrum of the whole line problem
is the union of the essential spectra of the half line problems; this is often referred to as the
decomposition method. So, in our situation, the two half line $m$ functions $m_{\pm}$ will both be meromorphic
on a neighborhood of $(-\infty, 0)$. In this situation, the negative eigenvalues of the whole line problem
would occur exactly at the $-t<0$ at which $m_+(-t)=-m_-(-t)$ or $m_+(-t)=m_-(-t)=\infty$; indeed, this is the
condition for the square integrable solutions on the half lines to arrive at $x=0$ with matching values.
Moreover, $m_{\pm}$ are still increasing on every subinterval of $(-\infty,0)$ that avoids the poles.
By looking at the possible scenarios, we can now deduce quickly that $m_{\pm}$ together can have at
most one pole on $(-\infty,0)$. In particular, Theorem \ref{T3.1} applies to both half lines,
so $H(x)=P_{\varphi(x)}$ for some function $\varphi(x)$ which is decreasing on both half lines
and then also decreasing overall if we add a suitable multiple of $\pi$ to it on one of the half lines.

Suppose now that we had $\varphi(-\infty)-\varphi(\infty)>\pi$, and here we may also assume
that $\varphi$ does not have jumps of size $\ge\pi$ because these could be replaced by jumps of smaller sizes
by removing these unnecessary multiples of $\pi$. As our first step, we then rotate, as in the last part of
the proof of Theorem \ref{TWW2}, in such a way that the new $\varphi$ ranges over an interval $(\alpha,\beta)
\supseteq [-\pi/2,\pi/2]$. This will not affect the property of $H$ of having non-negative spectrum because
acting on $H(x)$ by a rotation matrix will lead to a unitarily equivalent
(whole line) operator \cite[Theorem 7.2]{Rembook}. Since all jumps of $\varphi$ (if any) are of size $<\pi$,
we can then find an $a\in\R$ such that $\varphi(a-)<\pi/2$, $\varphi(a+)>-\pi/2$. Now Theorem \ref{TWW1}
(together with its mirror version for left half lines) shows that both half line problems, on $(-\infty, a)$
and $(a,\infty)$, have negative spectrum. However, as we just pointed out, this is impossible when the
whole line problem has non-negative spectrum.

The converse can be established by similar arguments.
If $\psi(-\infty)-\psi(\infty)\le\pi$, with $\psi=\varphi-\pi/2$,
then we can cut the whole line into two half lines in such a way that both half line coefficient functions
are as described in Theorem \ref{TWW1}. What we need to do here is cut at the unique point at which
$\psi$ crosses a value $\equiv 0\bmod\pi$, if there is one; if not, then we can cut at an arbitrary point.
Then we refer to Theorem \ref{TWW2} and its
analog for left half lines (and let's just say that we cut at $x=0$):
\begin{align}
\label{3.9}
\psi(0+)& =\alpha_+(-\infty), \quad\psi(\infty) =\alpha_+(0-) ,\\
\nonumber
\psi(0-)& =\alpha_-(-\infty), \quad \psi(-\infty) =\alpha_-(0-) .
\end{align}
The angles $\alpha_{\pm}$ again express the values of the $m$ functions:
$\pm m_{\pm}(-t) = \cot\alpha_{\pm}(-t)$. Note that $\alpha_+$ is decreasing on $(-\infty,0)$,
while $\alpha_-$ is increasing there.
When these monotonicity properties are combined with \eqref{3.9} and the information on the range of $\psi$,
then it will follow that $\alpha_{\pm}$ never take
the same value modulo $\pi$. (As usual, there is a trivial exceptional case here, when $H(x)\equiv P_{\beta}$,
which, also as usual, we leave to the reader.) So the whole line problem does not have negative eigenvalues,
and then the decomposition method finishes the proof.
\end{proof}
\section{The essential spectrum}
We will prove the following more general result, which will imply Theorems \ref{Tess}, \ref{T1.3}.
\begin{Theorem}
\label{T4.1}
Suppose that $H\in\mathcal C_+$, write $H(x)=P_{\varphi(x)}$, with $\varphi$ chosen as in Theorem \ref{TWW1}, and let
\[
A = \limsup_{x\to\infty} x(\varphi(x)-\varphi(\infty))
\]
(so $0\le A\le\infty$). Then
\[
\frac{1}{4A} \le \min \sigma_{ess} \le \frac{1}{A} .
\]
\end{Theorem}
Here we formally set $\min\emptyset =\infty$ and, as usual in such situations, $1/0=\infty$, $1/\infty=0$.

The presence of a gap between the upper and lower bounds
is unavoidable since $A$ does not provide enough information
to find the bottom of the essential spectrum exactly. This is possible, however,
if the limit exists; more generally, we have the following bound.
\begin{Theorem}
\label{T4.2}
Suppose that $H\in\mathcal C_+$, and let
\[
B = \liminf_{x\to\infty} x(\varphi(x)-\varphi(\infty)) .
\]
Then $\min\sigma_{ess}\le 1/(4B)$.
\end{Theorem}
\begin{proof}[Proof of Theorem \ref{T4.1}]
We first give an oscillation theoretic description of $T=\min\sigma_{ess}$ for $H\in\mathcal C_+$.
Clearly, $T$ is characterized by the pair of conditions $\dim E(0,t)<\infty$ for $t<T$,
$\dim E(0,t)=\infty$ for $t>T$. By Theorem \ref{T2.1}, this is equivalent to the corresponding conditions
\begin{equation}
\label{4.1}
\lim_{x\to\infty} \theta(x;t)<\infty \quad (0<t<T); \quad\quad
\lim_{x\to\infty} \theta(x;t)=\infty \quad (t>T)
\end{equation}
on the Pr{\"u}fer angle $\theta$ with $\theta(0;t)=0$, say.

We will again use the Pr{\"u}fer equation in the form \eqref{3.5}. As we observed earlier,
our only chance to come to rest is at the values $\psi(\infty)+n\pi$, so we only need to analyze what happens
when $\theta(x;t)$ comes close to one of these. Note that unlike in the previous section, the two angles
are now in contrary motion: $\psi$ decreases, while $\theta$ increases.

We start with the first inequality from Theorem \ref{T4.1}. For notational convenience, we assume that
$\psi(\infty)=0$; the general case can be reduced to this situation by applying a rotation,
as discussed in the last part of the proof of Theorem \ref{TWW2}. Actually, the agreement that $\psi(\infty)=0$
is not completely consistent with our earlier conventions on the range of $\psi$, but this discrepancy is
harmless; of course, we can always add multiples of $\pi$ to $\psi$.

We will then show that if
$0\le\psi(x)\le B/x$ ($x\ge a$) and $0<t<1/(4B)$, then the solution
$\theta(x)$ of
\begin{equation}
\label{4.2}
\theta' = t\sin^2(\theta-\psi(x)) ,
\end{equation}
with suitable initial value $\theta(a)=\theta_0<0$, will satisfy $\theta(x)<0$ for all $x>a$.
(This equation \eqref{4.2} is of course the same as \eqref{3.5}, but for positive spectral parameter $t$ now.)
This will establish that we are in the first case of \eqref{4.1}, and the desired inequality will follow
since $B>A$ can be taken arbitrarily close to $A$ if we make $a$ large enough.

Note also that it indeed suffices to discuss one specific initial value $\theta_0=\theta(a)$, and
it doesn't really matter what value $\theta_0$ we choose here: which alternative
of \eqref{4.1} holds will not depend on this value. Or in more concrete style, we can observe that if
a different initial value is chosen, then perhaps $\theta(x)$ will cross the value zero one more time,
but then on the next lap we will see the value $\theta_0$ again and the argument applies now.

We use the comparison principle, and we are interested in the range $\theta_0\le\theta\le 0$, so we estimate
the right-hand side of \eqref{4.2} from above by $t(\theta-B/x)^2$ and then consider the comparison
equation $\theta'_1 = t(\theta_1-B/x)^2$, $\theta_1(a)=\theta_0$. We have $\theta(x)\le\theta_1(x)$, so
it is now enough to show that $\theta_1(x)<0$ for all $x\ge a$.
In fact, by rescaling the $x$ variable, it suffices to consider the case $t=1$, so we will analyze the initial
value problem
\begin{equation}
\label{4.7}
\theta'_1 = \left( \theta_1 - \frac{C}{x} \right)^2, \quad \theta_1(a)=\theta_0 ,
\end{equation}
with $C=Bt<1/4$.

Introduce $\alpha = \theta_1-C/x$. Then the equation becomes
\begin{equation}
\label{4.3}
\alpha' = \alpha^2 + \frac{C}{x^2} .
\end{equation}
This is a Riccati equation, and the well known substitution $\alpha=-u'/u$ transforms
it into a Schr{\"o}dinger equation
\begin{equation}
\label{4.4}
-u''-\frac{C}{x^2}u = 0 ;
\end{equation}
more precisely, if we have a zero free solution $u$ of \eqref{4.4}, then $\alpha=-u'/u$ will solve \eqref{4.3}.
Now \eqref{4.4} is an Euler equation that can be solved explicitly by powers $x^p$, and a quick calculation
shows that here the admissible exponents are
\[
p_{\pm} = \frac{1}{2} \left( 1 \pm \sqrt{1-4C}\right) .
\]
We now take specifically the solution with the initial values $u(a)=0$, $u'(a)=1$. This will be computationally
convenient, but note that this actually corresponds formally to the initial value $\alpha(a)=-\infty$.
This will not be a problem because $\alpha(x)$ reaches finite values instantaneously for $x>a$, and, as we discussed,
$\theta$ can be assigned any negative initial value.

A straightforward calculation now shows that
\[
\alpha(x) = - \frac{p_+ a^{p_-}x^{p_+-1}-p_-a^{p_ +}x^{p_--1}}{a^{p_-}x^{p_+}-a^{p_+}x^{p_-}} .
\]
Since $p_++p_-=1$, we can rewrite this as
\[
\alpha(x) = - \frac{1}{a\xi} \frac{p_+\xi^d-p_-}{\xi^d-1}, \quad \xi = \frac{x}{a} \ge 1, \quad
d=p_+-p_-=\sqrt{1-4C} .
\]
We want to show that $\alpha(x)<-C/x$ for all $x\ge a$. This is certainly true initially, so we only
need to make sure that $\alpha(x)=-C/x$ can never happen. To confirm this, it suffices to set $y=\xi^d$
and then observe that the equation
\[
\frac{p_+y-p_-}{y-1} = C , \quad 0<C<1/4,
\]
has no solutions $y>1$.

The reader familiar with the spectral theory of Schr{\"o}dinger operators will undoubtedly have observed
that this part of the argument is powered by the well known fact
that the operator $-d^2/dx^2-C/x^2$ has no negative spectrum if $C<1/4$ (and this bound is sharp, there
will be infinite negative spectrum for $C>1/4$).

We now prove the upper bound $\min\sigma_{ess}\le 1/A$. Let $t>1/A$. In fact,
by rescaling, it will again suffice to treat the case $t=1$, $A>1$.
We must then show that if $\theta(a;t=1)=\theta_0<0$ at some $a>0$,
then $\theta(x)=0$ for some $x>a$ (which will then imply that $\theta(y)>0$
for $y>x$, and this is what we really need). The key point is that this must hold for any $a$, no
matter how large, for a given $\theta_0$.
The precise value of $\theta_0<0$, on the other hand, is again irrelevant.
Indeed, $\theta(x)$ will always approach zero;
the only question is if this value is reached in finite time.

So let $a>0$ be given, and let's also assume that $a$ is so large that $\psi(x)\le\pi/2$, say,
for $x\ge a$. There are arbitrarily large $b>a$ such that $\psi(b)\ge B/b$, and this
works for any $B<A$. Since $\psi$ is decreasing, we will then have $\psi(x)\ge B/b$ for
all $x\le b$. Thus
\[
\sin^2(\theta-\psi(x))\ge (1-\epsilon) \left( \theta-\frac{B}{b}\right)^2 ,
\]
and this will be valid on $x\in [a,b]$, as long as $\theta_0\le \theta\le 0$. Moreover,
we can achieve any $\epsilon>0$ here
if we take $\theta_0$ close enough to zero and $b$ large enough. In a moment, it will
turn out that we want $\epsilon < 1-1/B$; here, we of course assume that we took $B$ sufficiently
close to $A$, so that $B>1$ also.

We will then work with
\[
\theta'_1= (1-\epsilon) \left( \theta_1 - \frac{B}{b}\right)^2 , \quad \theta_1(a)=\theta_0
\]
as our comparison equation. This can be solved explicitly, and the perhaps most convenient way to do this is to
again introduce $\alpha=\theta_1-B/b$. Then $\alpha'=(1-\epsilon)\alpha^2$ and thus
\[
\alpha(x) = \frac{\theta_0-B/b}{1-(\theta_0-B/b)(1-\epsilon)(x-a)} \ge \frac{-1}{(1-\epsilon)(x-a)} .
\]
We have shown that
\[
\theta(x) \ge \frac{B}{b} - \frac{1}{(1-\epsilon)(x-a)}
\]
(at least as long as $\theta(x)\le 0$), and this lower bound can be made positive at $x=b$ since
$1/(1-\epsilon)<B$ and we can still take $b$ arbitrarily large.
\end{proof}
\begin{proof}[Proof of Theorem \ref{T4.2}]
This is very similar to what we just did, so we'll just give a brief sketch.
The comparison equation \eqref{4.7} also works as a lower bound if now $\psi(x)\ge C/x$,
$C<B$, and we introduce an additional factor $1-\epsilon$ on the right-hand side.
The analysis of this equation then proceeds exactly as in the first part of the previous proof.
The fact that the Schr{\"o}dinger operator $-d^2/dx^2-D/x^2$ has infinite negative spectrum
when $D>1/4$ will make the argument work.
\end{proof}

A classical, well known criterion for the absence of essential spectrum of a Schr{\"o}dinger operator
$\mathcal L=-d^2/dx^2+V(x)$ on $L^2(0,\infty)$ with $V\ge 0$, say, is \textit{Molchanov's criterion }\cite{Mol},
which says that $\sigma_{ess}(\mathcal L)=\emptyset$ if and only if
\begin{equation}
\label{mol}
\lim_{x\to\infty} \int_x^{x+d} V(t)\, dt =\infty \quad \textrm{\rm for all }d>0 .
\end{equation}
Schr{\"o}dinger equations $-y''+V(x)y=zy$ can be written as canonical systems, basically by running the
variation-of-constants method with the equation for $z=0$ taking the role of the unperturbed system;
see \cite[Section 1.3]{Rembook} for further details. To end up with a canonical
system $H\in\mathcal C_+$, we assume that $\mathcal L\ge 0$ also. The canonical system will be
of the form
\[
H_0(x) = \begin{pmatrix} p^2 & pq \\ pq & q^2 \end{pmatrix} ,
\]
and here $p,q$ solve $-y''+Vy=0$ and satisfy certain initial conditions (which are irrelevant for us and also depend
on the boundary condition at $x=0$ of $\mathcal L$). This coefficient function $H_0$ is not yet trace normed; to do this,
we need to pass to the new variable
\begin{equation}
\label{4.8}
X = \int_0^x \left( p^2(t) + q^2(t)\right) \, dt .
\end{equation}
We see that indeed $\det H_0(x)=0$, as guaranteed by Theorem \ref{TWW1}, so $H=P_{\varphi}$ with $\cot\varphi = p/q$.
Theorem \ref{TWW1} also shows that $M=\lim_{x\to\infty} p/q$ exists,
possibly after switching $p,q$, to avoid the scenario
where $M=\infty$. So if we introduce the new solution $f=p-Mq$, then $f/q\to 0$.

It will now be convenient to assume that $\min\sigma(\mathcal L)>0$. This is not really an extra assumption because
any energy can take over the role of $z=0$ in the transformation; what we do is deliberately choose an energy below
the spectrum. This has the technical advantage that we will then have an $L^2$ solution at this energy, and obviously,
in our situation, this must be the solution $f$ just constructed.

Theorem \ref{Tess} now says that $\sigma_{ess}=\emptyset$ if and only if $Xf/q\to 0$. There won't be any problems
with the zeros of $q$ here because $q$ has at most one, by (classical) oscillation
theory for Schr{\"o}dinger equations.
This condition can be given a more intuitive form. Constancy of the Wronskian $W=f'q-fq'=q^2(f/q)'$ implies that
\begin{equation}
\label{4.9}
\frac{f(x)}{q(x)} = -W\int_x^{\infty} \frac{dt}{q^2(t)} .
\end{equation}
By letting $f,q$ span the solution space, we see that every solution is either a multiple of $f$
or behaves asymptotically like a multiple of $q$. We then use \eqref{4.8}, \eqref{4.9} in the expression
$Xf/q$ and finally arrive at the following criterion.
\begin{Theorem}
\label{T4.3}
Let $\mathcal L=-d^2/dx^2+V(x)$ be a half line Schr{\"o}dinger operator that is bounded below, and fix an
$E_0<\min\sigma(\mathcal L)$. Let $q(x)$ be any solution of $-y''+Vy=E_0y$ with $q\notin L^2(0,\infty)$.
Then $\sigma_{ess}(\mathcal L)=\emptyset$ if and only if
\[
\lim_{x\to\infty} \int_0^x q^2(t)\, dt \int_x^{\infty} \frac{dt}{q^2(t)} = 0 .
\]
\end{Theorem}
It is not hard to show directly that this condition is equivalent to \eqref{mol}, but
we knew that already. So Theorem \ref{Tess} can be said to contain Molchanov's criterion as a special case,
but of course it is much more general because it applies to arbitrary canonical systems, not just the
ones that are Schr{\"o}dinger equations rewritten.
\section{Diagonal canonical systems and exponential orders}
Let $H\in\mathcal C_+$, and write $H(x)=P_{\varphi(x)}$, with $\varphi(x)$ chosen as in Theorem \ref{TWW1}.
In fact, it will be convenient now to also demand that $\varphi(x)$ is right-continuous.
Furthermore, we make the additional assumption that $-\pi/2<\varphi(x)\le \pi/2-\delta$ for some $\delta>0$.

Let $u$ be a solution of \eqref{can}, and introduce the new variable
\begin{equation}
\label{5.1}
t = -\tan\varphi(x) ;
\end{equation}
this is an increasing function of $x>0$, with range contained in $[-t_0,\infty)$, $t_0=-\tan\varphi(0+)$.
We want to rewrite \eqref{can} by using $t$ instead of $x$, so we would like to define $v(t)=u(x)$,
with $t$ and $x$ related by \eqref{5.1}, but here we must be careful since $t(x)$ can fail to be injective
and its range is not guaranteed to be an interval.

We address these technical issues as follows: gaps in the range of $t(x)$ result from jumps of $\varphi(x)$,
and if $\varphi(a)<\varphi(a-)$, then we simply set $v(t)=u(a)$ for $-\tan\varphi(a-)\le t< -\tan\varphi(a)$.
If, on the other hand, $\varphi(x)$ is constant on an interval $(a,b)$ (and this interval is maximal with this
property), then we set $v(t)=u(b)$ for $t=-\tan\varphi(a)$. There is no conflict between these definitions at
points at which both apply, and in all other cases, there is a unique $x$ with
$t=t(x)$, and the originally intended definition $v(t)=u(x)$ works. The function $v(t)$ thus defined is right-continuous
and of bounded variation, with jumps
precisely at the values $t$ that correspond to intervals of constancy of $\varphi$, and of course these
are exactly the singular intervals of the original system.

We now claim that we can rewrite the integrated form of \eqref{can},
\[
u(x)-u(0) = zJ\int_0^x P_{\varphi(y)}u(y)\, dy ,
\]
at a regular point $x$, as follows:
\begin{equation}
\label{5.2}
v(t) - v(t_0) = zJ \int_{(t_0,t]} \begin{pmatrix} 1 & -s \\ -s & s^2 \end{pmatrix} v(s)\, dw(s) ,
\end{equation}
and here $dw$ is a Borel measure on $(t_0,\infty)$ that is defined by the condition that
$(1+t^2)\,dw(t)$ is the image measure of $dx$ under the correspondence $x\mapsto t$.
Observe now that
\[
P_{\varphi(x)} = \frac{1}{1+t^2} \begin{pmatrix} 1 & -t \\ -t & t^2 \end{pmatrix} ,
\]
and this shows that we indeed obtain \eqref{5.2}, from the substitution rule.

It is perhaps also helpful to comment more explicitly on what happens here when $\varphi$ is either constant
on an interval or has a jump. In the first case, if $(a,b)$ is a singular interval, so
$\varphi(x)$ is constant on $[a,b)$, then $w$ will have the point mass $(1+t^2)w(\{t\})=b-a$ at $t=-\tan\varphi(a)$.
Therefore \eqref{5.2} will give $v$ a jump
\[
v(t) = \left( 1+z(b-a)JP_{\varphi(a)}\right) v(t-)
\]
at this point, which is exactly what the singular interval did to $u(x)$ across $(a,b)$.
If, on the other hand, $t\in (c,d)$ is an interval corresponding to a jump of $\varphi(x)$, then
$w((c,d))=0$, and this is consistent with the fact that $v$ is constant on this interval.

Next, we introduce
\[
y(t) = \begin{pmatrix} 1 & -t \\ 0 & 1 \end{pmatrix} v(t) .
\]
Then \eqref{5.2} is equivalent to
\begin{equation}
\label{5.3}
y(t)-y(t_0) = \int_{(t_0,t]} \begin{pmatrix} 0 & -ds \\ z\, dw(s) & 0 \end{pmatrix} y(s) .
\end{equation}
This we can confirm by a brute force calculation: by expressing everything in \eqref{5.3} in terms of $v$,
we see that we will obtain this equation if we can show that the matrix
$\bigl( \begin{smallmatrix} 0 & 1 \\ 0&0 \end{smallmatrix} \bigr)$ annihilates the vector
\[
(t-t_0)v(t_0)-\int_{t_0}^t v(s)\, ds +
z \int_{(t_0,t]} (t-s)J \begin{pmatrix} 1 & -s \\ -s & s^2 \end{pmatrix}v(s) \, dw(s) .
\]
To do this, write $t-s = \int_{[s,t]}du$ in the last term, change the order of integration in the resulting
double integral, and use \eqref{5.2}.

As our final transformation, we write $z=\zeta^2$ and introduce
\[
p(t) = \begin{pmatrix} \zeta & 0 \\ 0 & 1 \end{pmatrix} y(t) .
\]
Then \eqref{5.3} becomes
\begin{equation}
\label{5.5}
J(p(t)-p(t_0)) = -\zeta\int_{(t_0,t]} \begin{pmatrix} dw(s) & 0 \\ 0 & ds \end{pmatrix} p(s) ,
\end{equation}
and this is (almost) the promised diagonal canonical system that is associated with $H=P_{\varphi}$.
We can write this system in differential form if we pass to a new variable one more time.
More specifically, we let $T=w((t_0,t])+t-t_0$, so $dT$ is the trace of the coefficient matrix
from \eqref{5.5}, and then make $p$ a function of $T$. The transformation from $t$ to $T$ will correspond
exactly to the initial transformation of going from $x$ to $t$, except that we are now doing it in the
opposite direction. We will obtain the new coefficient function
\[
H_1(T) = \begin{pmatrix} h(T) & 0 \\ 0 & 1-h(T) \end{pmatrix} ,
\]
with $0\le h\le 1$. We give $H_1$ the required singular interval of type $0$ on the $T$ intervals corresponding
to the point masses of $dw$, and on the remaining set, we define $h$ by the condition that
$h\, dT$ is the image measure of $dw$ (and then $(1-h)\, dT$ will be related to $dt$ in the same way).

If $p(t_0,\zeta)$ is constant or has polynomial dependence on $\zeta$, then $p(T,\zeta)$ will be of order
at most $1$, and since $z=\zeta^2$, we now obtain Theorem \ref{Torder}(a) as an immediate consequence,
except that we made an additional assumption on the range of $\varphi(x)$ at the beginning of this section.
This, however, is easy to remove. Since we can compute the transfer matrix across an interval as a product
of transfer matrices across smaller subintervals, it will be enough to discuss the case when
$\varphi(0+)-\varphi(\infty)<\pi$. But then we can apply a rotation matrix, as discussed in the last part of the proof of Theorem \ref{TWW2},
to return to the situation already dealt with.

Part (c) then follows from
de~Branges's \cite{dB} well known formula for the exponential \textit{type }(not order) $\tau$
of a transfer matrix \cite[Theorem 4.26]{Rembook}, which can be computed as
\[
\tau = \int_0^T \sqrt{\det H_1(S)}\, dS = \int_0^T \sqrt{h(S)(1-h(S))}\, dS .
\]
Of course, if this is positive, then the order of $p$ will equal $1$ and thus $\textrm{ord}\: T(x;z)=1/2$.
So if this order is less than $1/2$, then $h=0$ or $1$ almost everywhere. This happens if and only if
$dw(t)$ is a purely singular measure and this is equivalent to $d\varphi$ being purely singular.

To prove part (b) of Theorem \ref{Torder}, we design suitable functions $A(z)=u_1(L;z)$, $C(z)=u_2(L;z)$,
with $u$ denoting the solution of \eqref{can} with $u(0;z)=e_1$. We will then obtain the coefficient function $H(x)$
from an inverse spectral theory result.

It is very easy to produce a canonical system with order $\nu=0$: a succession of finitely many singular intervals
will give us a polynomial transfer matrix. So we can focus on desired orders in the range $0<\nu<1/2$.
Let $\alpha=1/\nu > 2$, and define
\begin{equation}
\label{5.6}
A(z) = \prod_{n\ge 1} \left( 1 - \frac{z}{n^{\alpha}} \right) .
\end{equation}
This is the Hadamard product representation of an entire function with zeros $z_n=n^{\alpha}$,
and from the asymptotics of these it follows that $\textrm{ord}\: A = \nu$ \cite[Theorem 2.6.5]{Boas}.
We will then define a second function $C(z)$ in the same way, by giving it the zeros
$z_0=0$, $z_n = (n^{\alpha}+(n+1)^{\alpha})/2$, and $C'(0)>0$. Since these alternate with those of $A$,
it will then follow that $A-iC$ is a de~Branges function; see the discussion of \cite[Chapter VII]{Levin}.
By fundamental inverse spectral theory results \cite[Theorems 4.20, 5.2]{Rembook}, there will be a canonical
system on some interval $[0,L]$ whose transfer matrix $T(L;z)$ has $(A,C)^t$ as its first column if
$1/[(z+i)E(z)]\in H^2$, with $E=A-iC$. In our situation, it will be enough to verify that
\begin{equation}
\label{5.7}
\int_{-\infty}^{\infty} \frac{dt}{(1+t^2)|E(t)|^2} =
\int_{-\infty}^{\infty} \frac{dt}{(1+t^2)(A^2(t)+C^2(t))} < \infty .
\end{equation}
Both $|A(t)|$ and $|C(t)|$ are decreasing on $t<0$, so we can focus on $t>0$ here.
We will then estimate \eqref{5.6} to show that $|A(t)|$ can not get small
as long as we don't get close to its zeros, and of course $C$ will have the same property.
This will prove \eqref{5.7}.

The spectrum of the corresponding canonical system on $[0,L]$
with boundary condition $u_1(L)=0$ at $x=L$ is given by the zeros of $A$. As usual, we can then
view this as a half line problem, by setting $H(x)=P_{e_1}$ on $x>L$, and then $H\in\mathcal C_+$.
Thus \eqref{5.7} will also establish Theorem \ref{Torder}(b).

The argument to prove \eqref{5.7}
is quite routine, plus there are similar estimates available in the literature \cite[Chapter 4]{Boas},
so we will just give a sketch. Since $(n+1)^{\alpha}-n^{\alpha}\simeq \alpha n^{\alpha-1}$, it will be enough
to consider $t\ge 0$ with $|t-n^{\alpha}|\gtrsim n^{\alpha-1}$ for all $n\ge 1$. By replacing the sum by an integral,
it's then easy to see that
\[
\log |A(t)| = \sum_{n\ge 1} \log \left| 1- \frac{t}{n^{\alpha}} \right|
\]
will satisfy $\log |A(t)|\gtrsim I(\alpha)t^{1/\alpha}-O(\log t)$ for these $t$, with
\[
I(\alpha) = \int_0^{\infty} \log |1-s^{-\alpha}|\, ds .
\]
The monotonicity properties of $\log x$ imply that $I(\alpha)$ is strictly increasing, and
\begin{align*}
I(2) & = \lim_{L\to\infty} \int_0^L \log \frac{(s+1)|s-1|}{s^2}\, ds \\
& = \lim_{L\to\infty} \left( \int_1^{L+1} \log s\, ds + \int_{-1}^{L-1}\log |s| \, ds -2\int_0^L \log s\, ds \right) = 0 .
\end{align*}
So $I(\alpha)>0$ for $\alpha>2$, and \eqref{5.7} follows.

\end{document}